\documentclass[12pt,reqno]{amsart}
\usepackage{amsmath,amssymb,amsfonts,amsthm}
\usepackage{mathtools}
\usepackage{bm}
\usepackage[margin=1.25in]{geometry}
\usepackage{hyperref}
\usepackage[nameinlink,capitalise]{cleveref}
\usepackage{orcidlink}
\usepackage{hyperref}
\usepackage{tikz}

\usepackage{graphicx}
\usepackage{epstopdf}
\usepackage{amsfonts,amssymb,amsmath}
\usepackage{subfigure} 
\usepackage[utf8]{inputenc}
\usepackage{tikz}
\usepackage{xcolor}
\usepackage{hyperref}
\hypersetup{colorlinks=true,linkcolor=teal,citecolor=teal}
\usetikzlibrary{positioning}

\newtheorem{theorem}{Theorem}[section]
\newtheorem{proposition}[theorem]{Proposition}
\newtheorem{lemma}[theorem]{Lemma}
\newtheorem{corollary}[theorem]{Corollary}

\theoremstyle{definition}

\theoremstyle{remark}

\title[An exact structural identity and its normal form]{An exact structural identity and its normal form consequences in a modified Leslie--Gower model}

\author[Albarran-Garc\'ia, Alvarez-Ram\'irez]{Roberto Albarran-Garc\'ia\orcidlink{0009-0006-4591-7316}, Martha Alvarez-Ram\'irez\orcidlink{0000-0001-9187-1757}}
\address{Departamento de Matem\'aticas, UAM--Iztapalapa,
09310 Iztapalapa, Mexico City, Mexico}
\email{albarrangr74@live.com.mx,mar@xanum.uam.mx}

\author[Garc\'ia-Rivera]{Marco Polo Garc\'ia-Rivera\orcidlink{0009-0005-4774-3383}}
\address{Facultad de Ingenier\'ia, Universidad La Salle,
Mexico City, Mexico}
\email{mpgr86@gmail.com}

\begin{document}

\maketitle

\begin{abstract}
Zhao and Zhao (2026) established a codimension-four
Bogdanov--Takens singularity in a modified Leslie--Gower
predator--prey model through a recursive normal form computation. We
show that the nonlinear coefficients governing the restriction of
the reduced system to the singularity's distinguished
eigendirection are not independent. Instead, a single exact rational
identity, already present in the transformed equations,
generates the entire family of pure Taylor coefficients and proves
that their simultaneous vanishing is an exact all orders property
rather than a finite-order coincidence.
We further show that this structural identity propagates through the
planar Bogdanov--Takens normal form algorithm of Kuznetsov (2005).
Although the pure coefficients vanish identically, the corresponding
normal form coefficients are generated by mixed quadratic
interactions. This leads to an explicit characterization of the degenerate
Bogdanov--Takens locus and reduces the search for higher-order
Bogdanov--Takens degeneracies to a single explicit algebraic condition.
\end{abstract}

\section{Introduction} \label{sec:introduction}
Bogdanov--Takens (BT) singularities are organizing centers for
saddle--node, Hopf, and homoclinic bifurcations
\cite{Kuznetsov2004, GuckenheimerHolmes1983}. At a critical equilibrium, the linearization
has a nonsemisimple double-zero eigenvalue, and the local dynamics is
described by a planar nilpotent normal form
\cite{Bogdanov1981,Takens1974}. In suitable
coordinates, this normal form can be written as
\begin{equation}\label{eq:general-BT-normal form}
\dot w_0=w_1,
\qquad
\dot w_1=
\sum_{k\ge2}
\left(
a_k w_0^k
+
b_k w_0^{k-1}w_1
\right).
\end{equation}
The monomials depending only on \(w_0\) are referred to as
\emph{pure}, whereas those involving both variables are called
\emph{mixed}. Accordingly, \(a_k\) and \(b_k\) denote the pure and
mixed coefficients of the normal form, respectively. Successive
Bogdanov--Takens degeneracies are determined by the vanishing of
specific normal form coefficients. Practical recursive formulas for
computing these coefficients were developed by Kuznetsov
\cite{Kuznetsov2005} through the Fredholm solvability conditions
associated with the homological equations.

Predator--prey models constitute an important source of such
degenerate equilibria. In particular, modified Leslie--Gower models
incorporating Allee effects, nonlinear functional responses, or
alternative food resources exhibit rich bifurcation structures,
including Bogdanov--Takens singularities of codimension higher than
two. Recently, Zhaoand Zhao \cite{ZhaoZhao2026} studied a modified
Leslie--Gower predator--prey model with an additive Allee effect and
a Beverton--Holt type reproduction function. Using a recursive
normal form computation, they identified Bogdanov--Takens
singularities of codimensions two, three, and four together with the
corresponding local bifurcation scenarios.

The recursive nature of the normal form computation raises a natural
structural question: are the higher-order coefficients genuinely
independent, or do they reflect an algebraic structure already present
before the successive near-identity transformations of the Takens
normalization are performed? The present paper addresses this question
for the modified Leslie--Gower model considered by Zhao and Zhao
\cite{ZhaoZhao2026}.

Rather than reproducing their recursive normal form computation, we
first investigate the transformed equations obtained after the nilpotent reduction. Our main result is an exact rational
identity for the restriction of the transformed second component to
the distinguished eigendirection of the nilpotent linearization. This
identity yields a closed formula for the complete family of pure
Taylor coefficients and shows that all higher-order pure terms are
generated by a single structural factor. Consequently, their
simultaneous cancellation is an exact all orders property of the
nilpotent reduction rather than an accidental feature of the first
few terms of the Taylor expansion. In particular, the identity
provides the structural mechanism underlying the cancellations
observed in the recursive computation.

The structural identity is not confined to the linearly reduced
system. We subsequently incorporate it into the planar
Bogdanov--Takens normal form procedure of Kuznetsov
\cite{Kuznetsov2005}. Although the pure Taylor coefficients may
vanish simultaneously, the corresponding normal form coefficients
need not vanish, since mixed nonlinear interactions contribute during
the near identity transformations. This yields explicit expressions
for the first normal form coefficients, completely characterizes the
distinguished case in which all higher-order pure coefficients
disappear, and reduces the search for higher-order degeneracies to a
single explicit algebraic condition.
The overall strategy of the paper is summarized in Figure~\ref{fig:strategy}.
This approach separates the structural information already contained
in the nilpotent reduction from the subsequent normal form computation.
The first part of the paper establishes the exact identity and its
all orders consequences, whereas the second part shows how this
information simplifies the computation of the Bogdanov--Takens
normal form coefficients.
\begin{figure}[hbt]
\centering
\begin{tikzpicture}[
>=stealth,
node distance=1.55cm,
every node/.style={font=\small},
box/.style={
draw,
rounded corners,
minimum width=3.6cm,
minimum height=0.9cm,
align=center},
arrow/.style={->,thick}
]

\node[box] (A1)
{Modified Leslie--Gower model};

\node[box,below of=A1] (A2)
{Linear nilpotent reduction};

\node[box,below of=A2] (A3)
{Recursive Takens\\ normal form computation};

\node[box,below of=A3] (A4)
{Pure and mixed coefficients\\ obtained order by order};

\draw[arrow] (A1)--(A2);
\draw[arrow] (A2)--(A3);
\draw[arrow] (A3)--(A4);

\node[box,right=4.6cm of A2] (B1)
{Exact rational identity};

\node[box,below of=B1] (B2)
{Restriction to the\\ distinguished eigendirection};

\node[box,below of=B2] (B3)
{Closed Taylor expansion};

\node[box,below of=B3] (B4)
{Pure coefficients\\ obtained at all orders};

\node[box,below of=B4] (B5)
{Kuznetsov normal form\\ coefficients};

\node[box,below of=B5] (B6)
{Degenerate Bogdanov--Takens\\ locus and fourth-order condition};

\draw[arrow] (A2)--(B1);
\draw[arrow] (B1)--(B2);
\draw[arrow] (B2)--(B3);
\draw[arrow] (B3)--(B4);
\draw[arrow] (B4)--(B5);
\draw[arrow] (B5)--(B6);

\node[above=0.3cm of A1]
{\textbf{Recursive approach}};

\node[above=0.3cm of B1]
{\textbf{Present approach}};
\end{tikzpicture}

\caption{Comparison between the recursive normal form computation and
the approach developed in this paper. The recursive procedure
determines the normal form coefficients order by order. In contrast,
the present approach first derives an exact rational identity for the
restriction of the system~\eqref{eq:nilpotent-system} to
the distinguished eigendirection, then generates the complete family
of pure Taylor coefficients through a single Taylor expansion, and
finally incorporates this structural information into the Kuznetsov
normal form computation.}\label{fig:strategy}
\end{figure}
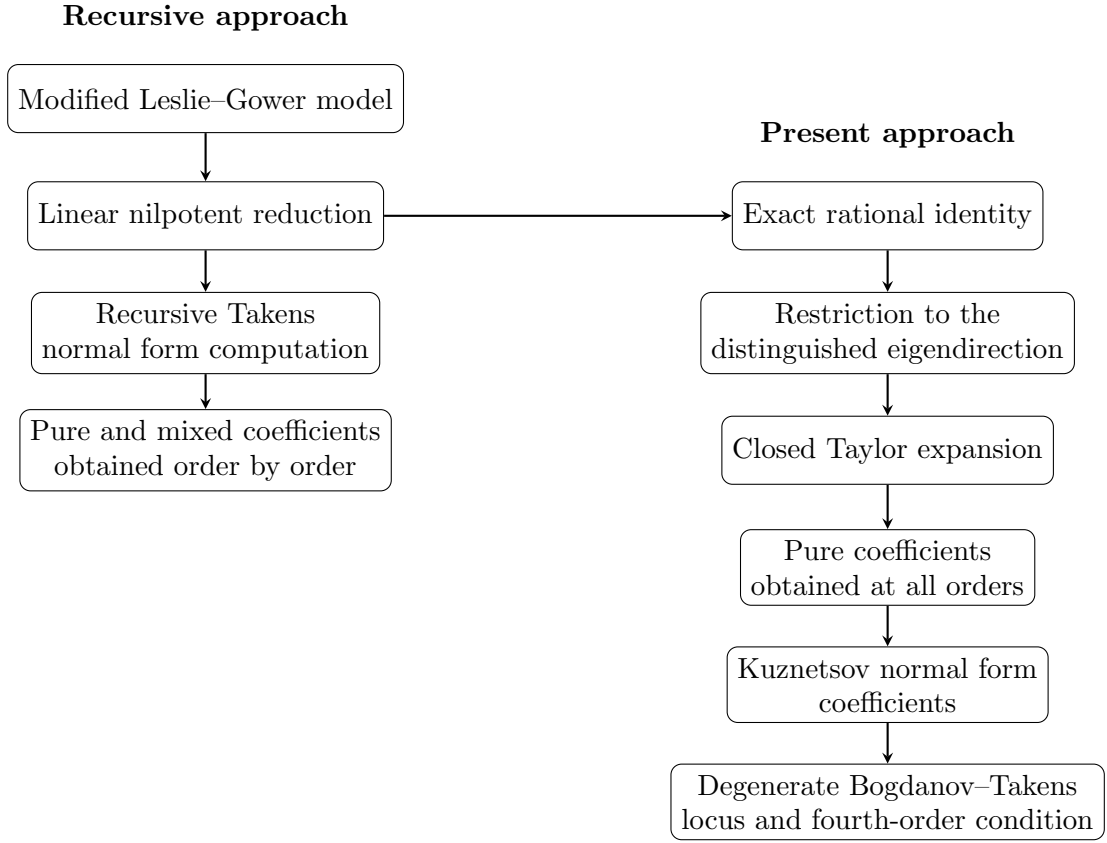

The paper is organized as follows. Section~\ref{sec:nilpotent}
derives the explicit nilpotent reduction of the modified
Leslie--Gower model. Section~\ref{sec:identity} establishes the exact
structural identity and obtains a closed formula for the complete
family of pure Taylor coefficients. Section~\ref{sec:normal form}
discusses the relation between this structural identity and the
classical Takens normal form approach, clarifying the role of the
pure coefficients in the recursive normalization. Finally,
Section~\ref{sec:kuznetsov} derives the principal consequences of the
structural identity for the smooth Bogdanov--Takens normal form of
Kuznetsov, including explicit expressions for the first normal form
coefficients, the characterization of the degenerate Bogdanov--Takens locus, and the
explicit algebraic condition governing higher-order degeneracies. The
coefficient polynomials defining the latter condition are collected
in Appendix~\ref{app:P4}.

\section{Nilpotent reduction}\label{sec:nilpotent}
We begin with the dimensionless predator--prey model introduced by Zhao and Zhao~\cite{ZhaoZhao2026}, which serves as the starting point
for our analysis:
\begin{equation}\label{eq:model}
\begin{aligned}
\dot x &= x(1-x)-\frac{\alpha x}{x+\beta}-rxy,\\
\dot y &= \frac{\delta y}{1+\eta y}\left(1-\frac{y}{x+c}\right).
\end{aligned}
\end{equation}
Here \(x\) and \(y\) denote the prey and predator densities,
respectively. The parameter \(\alpha\) measures the strength of the
additive Allee effect, \(\beta\) is the corresponding threshold,
\(r\) is the predation coefficient, \(\delta\) is the intrinsic
predator growth rate, \(\eta\) characterizes the Beverton-Holt
density dependence, and \(c\) represents the availability of
alternative food resources. Unless otherwise stated, all parameters are assumed to be positive.
We shall also consider the boundary case $\beta=0$, which plays an
important role in the structural analysis below.

We restrict our attention to positive equilibria of \eqref{eq:model}, since these
are the equilibria considered by Zhao and Zhao \cite{ZhaoZhao2026} in their
analysis of the codimension-four Bogdanov--Takens degeneracy.
Our aim is not to reproduce their normal form analysis, but to
identify an exact structural property of the associated nilpotent
reduction.

A positive equilibrium has the form
\[
E=(x_0,x_0+c),
\]
and the equilibrium equation is
\begin{equation}\label{eq:equilibrium-polynomial}
(r+1)x_0^2+(\beta+r\beta+rc-1)x_0+\beta rc-\beta+\alpha=0.
\end{equation}
Solving \eqref{eq:equilibrium-polynomial} for $\alpha$ gives
\begin{equation}\label{eq:alpha-equilibrium}
\alpha=-(r+1)x_0^2-(\beta+r\beta+rc-1)x_0-\beta rc+\beta.
\end{equation}

At $E$, the Jacobian matrix is
\begin{equation}\label{eq:J-equilibrium}
J(E)=
\begin{pmatrix}
-\dfrac{x_0(\beta+cr+rx_0+2x_0-1)}{\beta+x_0} & -rx_0\\[1.1ex]
\dfrac{\delta}{1+\eta(x_0+c)} & -\dfrac{\delta}{1+\eta(x_0+c)}
\end{pmatrix}.
\end{equation}
Imposing
$\operatorname{tr}J(E)=0$, and  $\det J(E)=0$,
yields
\begin{equation}\label{eq:BT-parametrization}
\begin{aligned}
\beta&=\frac{1-cr-2(r+1)x_0}{r+1},\qquad
\delta=rx_0\bigl(1+\eta(x_0+c)\bigr),\\
\alpha&=\frac{\bigl(cr+(r+1)x_0-1\bigr)^2}{r+1}.
\end{aligned}
\end{equation}
Hence every Bogdanov--Takens equilibrium is parametrized by the free
variables \(r,c,\eta,x_0\), whereas the remaining parameters
\(\alpha,\beta,\delta\) are determined by \eqref{eq:BT-parametrization}. Throughout the
paper we work on this Bogdanov--Takens parameter manifold.

Once the relations \eqref{eq:BT-parametrization} hold,
\begin{equation}\label{eq:nilpotent-J}
J(E)=rx_0
\begin{pmatrix}
1& &-1\\
1& &-1
\end{pmatrix},
\qquad
J(E)^2=0,
\end{equation}
so the linearization is a nonzero nilpotent matrix of index two.
This nilpotent structure is the starting point of our analysis.
Accordingly, we transform system~\eqref{eq:model} into the canonical
Jordan form associated with \(J(E)\). We first translate the
equilibrium to the origin by introducing
\[
X=x-x_0,
\qquad
Y=y-(x_0+c).
\]
In the translated coordinates, the Jacobian matrix has a double-zero
eigenvalue with a single Jordan block. Let
\[
q_0=
\begin{pmatrix}
1\\
1
\end{pmatrix},
\qquad
q_1=
\begin{pmatrix}
\dfrac{1}{rx_0}\\[0.4em]
0
\end{pmatrix}
\]
be vectors satisfying
\[
J(E)q_0=0,
\qquad
J(E)q_1=q_0.
\]
Thus, \((q_0,q_1)\) forms a Jordan chain for the nilpotent matrix
\(J(E)\). Introducing the matrix
\[
P=(q_0\ q_1)
=
\begin{pmatrix}
1& & \dfrac{1}{rx_0}\\[0.5em]
1& & 0
\end{pmatrix},
\]
one obtains
\[
P^{-1}J(E)P=
\begin{pmatrix}
0&1\\
0&0
\end{pmatrix},
\]
which is the canonical nilpotent Jordan form. 
This canonical form naturally motivates introducing coordinates
\((u,v)\) adapted to the Jordan chain \((q_0,q_1)\). Specifically,
the translated variables \((X,Y)\) are expressed in this basis as
\[
\begin{pmatrix}
X\\
Y
\end{pmatrix}
=
P
\begin{pmatrix}
u\\
v
\end{pmatrix},
\]
that is,
\begin{equation}\label{eq:Jordan-change}
X=u+\frac{v}{rx_0},
\qquad
Y=u,
\end{equation}
or, equivalently,
\[
u=Y,
\qquad
v=rx_0(X-Y).
\]
Recalling that $X=x-x_0$ and $Y=y-(x_0+c)$, the original variables
are recovered from
\begin{equation}\label{eq:original-variables}
x=x_0+u+\frac{v}{rx_0},
\qquad
y=x_0+c+u.
\end{equation}
Then, the transformed system \eqref{eq:model} takes the form
\begin{equation}\label{eq:nilpotent-system}
\dot u=v+F(u,v),
\qquad
\dot v=G(u,v),
\end{equation}
where
$F(u,v),\,G(u,v)=\mathcal{O}\!\left(\|(u,v)\|^2\right).$
The functions \(F\) and \(G\) collect the nonlinear terms remaining
after the linear nilpotent reduction.
Therefore, the nonlinear dynamics near the equilibrium is completely
encoded in the functions \(F\) and \(G\). The next section shows that
the restriction of \(G\) to the eigendirection generated by \(q_0\)
possesses a remarkable exact structure. Rather than being determined
independently order by order, all pure Taylor coefficients arise from
a single rational identity. This identity constitutes the central result of the first part of the
paper and provides the structural mechanism underlying the successive
cancellations observed in the recursive normal form computation.
For later use in Section~\ref{sec:identity}, we introduce
the two quantities
\begin{equation}\label{eq:DeltaTheta}
\Delta=cr+(r+1)x_0-1,
\qquad
\Theta=cr+2(r+1)x_0-1=\Delta+(r+1)x_0.
\end{equation}
Combining these definitions with the parametrization
\eqref{eq:BT-parametrization} gives
\begin{equation}\label{eq:Theta-beta}
\Theta=-(r+1)\beta.
\end{equation}
In particular, $\Theta$ vanishes exactly when $\beta=0$.
The next section shows that $\Theta$ is the common structural factor
appearing in every higher-order pure coefficient.

\section{Exact structural identity} \label{sec:identity}
The objective of this section is to derive an exact rational expression
for the restriction of the transformed system \eqref{eq:nilpotent-system} to the
distinguished eigendirection generated by \(q_0\). Unlike the
recursive Takens normal form procedure, the derivation is exact and
does not rely on a finite-order Taylor expansion. 
As a consequence, the complete family of pure Taylor coefficients is
determined simultaneously from the system~\eqref{eq:nilpotent-system}.

Let
\begin{equation*}
H(x):=f(x,x+c),
\end{equation*}
where $f$ is the first component of \eqref{eq:model}. Then
\begin{equation}\label{eq:F-explicit}
H(x)
=
x\bigl(1-cr-(1+r)x\bigr)
-\frac{\alpha x}{x+\beta}.
\end{equation}
Using \eqref{eq:BT-parametrization} and
\eqref{eq:Theta-beta}, we obtain
\begin{equation}\label{eq:BT-identities}
\alpha=\frac{\Delta^{2}}{1+r},
\qquad \beta=-\frac{\Theta}{1+r},
\qquad x_0+\beta=-\frac{\Delta}{1+r}.
\end{equation}

\begin{lemma}\label{lem:doublezero}
The function \(H\) satisfies
\[
H(x_0)=0,
\qquad
H'(x_0)=0.
\]
Consequently, \(x_0\) is a zero of multiplicity at least two of
\(H\).
\end{lemma}

\begin{proof}
Since $E=(x_0,x_0+c)$ is an equilibrium, $H(x_0)=f(x_0,x_0+c)=0.$
Differentiating \eqref{eq:F-explicit} yields
\[
H'(x)
=
1-cr-2(1+r)x-\frac{\alpha\beta}{(x+\beta)^2}.
\]

Using \eqref{eq:BT-identities},
\[
\frac{\alpha\beta}{(x_0+\beta)^2}
=
\frac{\dfrac{\Delta^2}{1+r}\,\beta}
{\dfrac{\Delta^2}{(1+r)^2}}=(1+r)\beta.\]

Since
\[
(1+r)\beta
=
1-cr-2(1+r)x_0,
\]
we obtain
\[H'(x_0)=1-cr-2(1+r)x_0-(1+r)\beta=0.\]
\end{proof}
Lemma~\ref{lem:doublezero} shows that the quadratic factor appearing
later in the restriction of the transformed system \eqref{eq:nilpotent-system} is not an
artifact of a Taylor expansion. It is a direct consequence of the
Bogdanov--Takens conditions, which force \(x_0\) to be a double zero
of the scalar function \(H\).

The previous observation suggests looking for an exact factorization
of \(H\), rather than expanding it order by order.
\begin{proposition}\label{prop:F-factorization}
For \(x=x_0+u\), the function \(H\) admits the exact factorization
\begin{equation}\label{eq:F-factorization}
H(x_0+u)
=
\frac{(1+r)^2u^2(x_0+u)}
{\Delta-(1+r)u}.
\end{equation}
\end{proposition}

\begin{proof}
Substituting the identities
\eqref{eq:BT-identities} into
\eqref{eq:F-explicit} gives
$$
H(x)=x\left[1-cr-(1+r)x-\frac{\Delta^2}{(1+r)(x+\beta)}\right].$$
Since $1-cr=(1+r)x_0-\Delta,$
setting $x=x_0+u$ yields
$$
1-cr-(1+r)(x_0+u)
=
-\Delta-(1+r)u.$$
Moreover, by
\eqref{eq:BT-identities},
\[
x_0+u+\beta=u-\frac{\Delta}{1+r}=-\frac{\Delta-(1+r)u}{1+r}.
\]
Therefore,
\[
H(x_0+u)=(x_0+u)
\left[
-\Delta-(1+r)u
+
\frac{\Delta^2}
{\Delta-(1+r)u}
\right].
\]
Finally, the elementary identity
\[
-\Delta-(1+r)u
+
\frac{\Delta^2}
{\Delta-(1+r)u}
=
\frac{(1+r)^2u^2}
{\Delta-(1+r)u}
\] gives \eqref{eq:F-factorization}.
\end{proof}

Proposition~\ref{prop:F-factorization} shows that the double-zero
property established in Lemma~\ref{lem:doublezero} admits an exact
rational representation. We now transfer this identity to the
transformed system \eqref{eq:nilpotent-system}. The resulting expression will determine the
entire hierarchy of pure Taylor coefficients simultaneously.
Recall from \eqref{eq:Jordan-change} that
\[
X=u+\frac{v}{rx_0},
\qquad
Y=u,
\]
or equivalently,
\[
x=x_0+u+\frac{v}{rx_0},
\qquad
y=x_0+c+u.
\]

We now use these coordinates to restrict system~\eqref{eq:nilpotent-system}
to the eigendirection generated by \(q_0\), that is, to the line \(v=0\).

\begin{proposition}\label{prop:restriction}
On the double zero parameter set, the restriction of the transformed
system~\eqref{eq:nilpotent-system} to the eigendirection generated by \(q_0\) is
\begin{equation}\label{eq:restriction}
G(u,0) =
\frac{r(1+r)^2x_0u^2(x_0+u)}
{\Delta-(1+r)u}.
\end{equation}
\end{proposition}

\begin{proof}
Restricting the inverse change of
coordinates~\eqref{eq:original-variables} to the line \(v=0\) yields
\[
x=x_0+u,
\qquad
y=x_0+c+u=x+c.
\]
Hence
\[
\frac{\delta y}{1+\eta y}
\left(
1-\frac{y}{x+c}
\right)\equiv0,
\]
so that the second equation of the original system vanishes
identically along this line. Since
\[
v=rx_0(X-Y),
\]
it follows that
\[
\dot v=rx_0(\dot X-\dot Y),
\]
and therefore
\[
G(u,0)=rx_0H(x_0+u).
\]
Substituting the factorization of \(H\) obtained in
Proposition~\ref{prop:F-factorization} immediately gives
\eqref{eq:restriction}.
\end{proof}
Consequently, the Taylor coefficients of \(G(u,0)\) are precisely the
coefficients of the pure powers of \(u\) in the second component of
system~\eqref{eq:nilpotent-system}. Unlike the
recursive Takens normal form procedure, which determines these
coefficients successively through near-identity transformations, the
exact rational identity~\eqref{eq:restriction} determines the entire
family simultaneously through a single Taylor expansion.

\begin{corollary}\label{cor:taylor-coefficients}
Assume that $\Delta \neq 0$. Then the Taylor expansion of $G(u,0)$
at $u=0$ is of the form
\[
G(u,0) = \sum_{k\ge 2} a_k u^k,
\]
where, for a function $\varphi$ analytic at $u=0$, $[u^k]\varphi(u)$
denotes the coefficient of $u^k$ in its Taylor expansion, so that
\[
a_2 = \frac{r(1+r)^2x_0^2}{\Delta},
\]
and, for every $k\ge3$,
\begin{equation}
\label{eq:ak-formula}
a_k = [u^k]G(u,0) = \frac{r(1+r)^{k-1}x_0\Theta}{\Delta^{k-1}}.
\end{equation}
\end{corollary}

\begin{proof}
Using the geometric expansion
\[
\frac1{\Delta-(1+r)u}
=
\frac1{\Delta}
\sum_{j=0}^{\infty}
\left(
\frac{(1+r)u}{\Delta}
\right)^j,
\qquad
\left|
\frac{(1+r)u}{\Delta}
\right|<1,
\]
equation~\eqref{eq:restriction} becomes
\[
G(u,0)
=
\frac{r(1+r)^2x_0}{\Delta}
u^2(x_0+u)
\sum_{j=0}^{\infty}
\left(
\frac{(1+r)u}{\Delta}
\right)^j.
\]
The coefficient of \(u^2\) is immediately
\[a_2=\frac{r(1+r)^2x_0^2}{\Delta}.\]

For \(k\ge3\), the coefficient of \(u^k\) receives two contributions,
coming respectively from the factors \(x_0\) and \(u\) in
\((x_0+u)\). Their sum is
\[
\frac{r(1+r)^{k-1}x_0}
{\Delta^{k-1}}
\bigl((1+r)x_0+\Delta\bigr).
\]
Since
\[
(1+r)x_0+\Delta
=
cr+2(1+r)x_0-1
=
\Theta,
\]
formula~\eqref{eq:ak-formula} follows.
\end{proof}

\begin{corollary}\label{cor:beta-zeroo}
If $\beta=0$, then
\[
[u^k]G(u,0)=0,
\qquad
k\ge3.
\]
Consequently,
\[
G(u,0)
=
\frac{r(1+r)^2x_0^2}{\Delta}u^2,
\]
that is, the restriction of the system \eqref{eq:nilpotent-system} to the
eigendirection generated by \(q_0\) becomes exactly quadratic.
\end{corollary}

\begin{proof}
By \eqref{eq:BT-identities}, $\Theta=-(1+r)\beta$, so the
condition $\beta=0$ is equivalent to $\Theta=0$. Corollary~\ref{cor:taylor-coefficients}
then implies that every pure coefficient of degree at least three
vanishes simultaneously, leaving $G(u,0)$ exactly quadratic.

This phenomenon can also be seen directly from the exact rational
identity \eqref{eq:restriction}, without invoking its Taylor
expansion. Indeed, when $\beta=0$, the parametrization
\eqref{eq:BT-parametrization} gives
$cr=1-2(r+1)x_0,$
and substituting this relation into
$\Delta=cr+(r+1)x_0-1$
yields $\Delta=-(1+r)x_0.$
Consequently,
\[
\frac{x_0+u}{\Delta-(1+r)u}
=\frac{x_0+u}{-(1+r)x_0-(1+r)u}
=-\frac{1}{1+r},
\]
a genuine constant, since the factor $x_0+u$ cancels exactly
against the denominator. Substituting back into
\eqref{eq:restriction} gives
\[
G(u,0)=r(1+r)^2x_0\,u^2\cdot\left(-\frac{1}{1+r}\right)
=-r(1+r)x_0\,u^2,
\]
which confirms, at the level of the exact rational function itself
rather than order by order, that $G(u,0)$ collapses identically to
a quadratic when $\beta=0$.
\end{proof}

Proposition~\ref{prop:restriction} and its corollaries show that all pure
nonlinear coefficients of degree at least three are proportional to
the same structural factor~\(\Theta\).
The Bogdanov--Takens parametrization
\eqref{eq:BT-parametrization} coincides with that obtained by Zhao and Zhao
\cite{ZhaoZhao2026}. In particular,
\[
\beta=\frac{1-cr-2(r+1)x_0}{r+1},
\]
so that the distinguished case $\beta=0$ is equivalent to
\[
x_0=\frac{1-cr}{2(r+1)}.
\]
By Corollary~\ref{cor:beta-zeroo}, this parameter value is
characterized by the simultaneous vanishing of all pure nonlinear
coefficients of degree at least three.
To illustrate the previous results, we consider the parameter values
identified by Zhao and Zhao \cite{ZhaoZhao2026} in their analysis of a
codimension-four Bogdanov--Takens degeneracy:
\[
r=0.6,\qquad
c=0.2,\qquad
\eta=2,\qquad
x_0=0.275,
\]
for which the Bogdanov--Takens parametrization
\eqref{eq:BT-parametrization} gives
\[
\alpha=0.121,\qquad
\beta=0,\qquad
\delta=0.32175,\qquad
\Delta=-0.44\neq0.
\]
Hence Corollary~\ref{cor:taylor-coefficients} applies and the exact
restriction~\eqref{eq:restriction} reduces to
\[
G(u,0)=\frac12\,r(cr-1)u^2=-0.264\,u^2,
\]
confirming that every pure nonlinear coefficient of degree at least
three vanishes identically.

Finally, all the identities established in this section, including the
double-zero property of Lemma~\ref{lem:doublezero}, the exact
factorization~\eqref{eq:F-factorization}, the restriction
\eqref{eq:restriction}, and the coefficient
formula~\eqref{eq:ak-formula}, were independently verified by exact
symbolic computations in \textit{Mathematica}. In particular,
\eqref{eq:ak-formula} was checked term by term against the Taylor
expansion of \eqref{eq:restriction} through order twelve.
Although the structural identity completely determines the pure Taylor
coefficients of the linearly reduced system~\eqref{eq:nilpotent-system}, the coefficients of
the smooth Bogdanov--Takens normal form are affected by the successive
near-identity transformations. The next section explains this
distinction.

\section{Relation with the degenerate Bogdanov--Takens normal form}
\label{sec:normal form}
The structural identity established in Section~\ref{sec:identity} should be interpreted
as a structural
property of the nilpotent reduction rather than as a complete
derivation of the degenerate Bogdanov--Takens normal form.

The analysis in~\cite{ZhaoZhao2026} determines the
codimension-two, -three, and -four degeneracies through successive
near-identity transformations leading to the Takens normal form and the
corresponding normal form coefficients. In contrast, the present
approach stops before these transformations are performed and focuses
instead on the structure of the nilpotent system itself.
%
%
Proposition~\ref{prop:restriction} shows that the restriction of the
system~\eqref{eq:nilpotent-system} to the eigendirection generated by $q_0$
is an exact rational function. Consequently, every pure nonlinear
coefficient of order at least three is proportional to the same
structural factor $\Theta$. Since $\Theta=-(1+r)\beta$, by
\eqref{eq:Theta-beta}, all these coefficients vanish simultaneously
when $\beta=0$. This cancellation should not be identified with the
complete codimension-four Bogdanov--Takens conditions obtained by Zhao and Zhao
\cite{ZhaoZhao2026}, which also involve mixed normal form
coefficients generated during the Takens normalization. Rather, the
identity established in Proposition~\ref{prop:restriction} is an
exact all orders property of the linearly reduced system. While the
Takens normalization determines the complete local dynamics, the
exact identity established in Section~\ref{sec:identity} shows that all pure nonlinear
coefficients of degree at least three arise from a single exact
rational expression associated with the linearly reduced system.
This reveals their common structural origin and provides a closed-form
expression for them at every order.

\section{Consequences for the Kuznetsov normal form}
\label{sec:kuznetsov}
The exact structural identity established in Section~\ref{sec:identity} completely
determines the pure Taylor coefficients of 
system~\eqref{eq:nilpotent-system}. A natural question is whether the
corresponding coefficients of the smooth Bogdanov--Takens normal form
also vanish. The answer is negative. The recursive normal form
procedure of Kuznetsov~\cite{Kuznetsov2005} combines both pure and
mixed Taylor coefficients through successive near-identity
transformations. Consequently, the normal form coefficients cannot, in
general, be inferred solely from the pure Taylor coefficients. The
purpose of this section is to make this relation explicit.

\subsection{Taylor coefficients and normal form coefficients}
Following the recursive procedure in~\cite{Kuznetsov2005}, the normal form coefficients are
obtained recursively from the Taylor coefficients of
system~\eqref{eq:nilpotent-system} by solving the associated
homological equations. Since the intermediate recursive calculations
and auxiliary coefficients play no role in the present analysis, we
report directly the resulting closed expressions for the normal form
coefficients required below.

We write the Taylor expansion of system~\eqref{eq:nilpotent-system} in the form
\begin{equation}\label{eq:Taylor-Kuznetsov}
\begin{aligned}
\dot u &= v+ \sum_{j+k\geq2} \frac{a_{jk}}{j!\,k!}\,u^jv^k,\qquad 
\dot v &=
\sum_{j+k\geq2}
\frac{b_{jk}}{j!\,k!}\,u^jv^k.
\end{aligned}
\end{equation}
Thus, \(a_{jk}\) and \(b_{jk}\) are the Taylor coefficients of the
linearly reduced system. They should not be confused with the
coefficients of its smooth normal form.

To distinguish the two families, we denote the normal form
coefficients by \(A_k\) and \(B_k\). Following
\cite{Kuznetsov2005}, system~\eqref{eq:Taylor-Kuznetsov} can be
transformed by smooth near-identity coordinate changes into
\begin{equation}\label{eq:Kuznetsov-normal form}
\begin{aligned}
\dot w_0&=w_1,\\
\dot w_1
&=
A_2w_0^2+B_2w_0w_1
+A_3w_0^3+B_3w_0^2w_1 +A_4w_0^4+B_4w_0^3w_1
+\mathcal{O}\!\left(\|(w_0,w_1)\|^5\right).
\end{aligned}
\end{equation}
Unlike the coefficients \(a_{jk}\) and \(b_{jk}\), the quantities
\(A_k\) and \(B_k\) incorporate contributions produced by both the
pure and mixed terms during the normal form transformation.

\begin{proposition}\label{prop:Kuznetsov-pure}
For the reduced predator--prey system,
\[
a_{k0}=0,
\qquad k\geq2.
\]
Moreover,
\[
b_{20}=
\frac{2r(1+r)^2x_0^2}{\Delta},
\]
and, for every \(k\geq3\),
\[
b_{k0}
=
k!\,
\frac{r(1+r)^{k-1}x_0\Theta}{\Delta^{k-1}}.
\]
Consequently,
\[
b_{k0}=0,
\qquad k\geq3,
\]
whenever \(\Theta=0\).
\end{proposition}

\begin{proof}
From the inverse change of coordinates,
\[
x=x_0+u+\frac{v}{rx_0},
\qquad
y=x_0+c+u.
\]
Hence, on the line \(v=0\),
\[
x=x_0+u,
\qquad
y=x_0+c+u=x+c.
\]
Hence the second equation of the original system vanishes
identically. Since \(u=Y\), it follows that \(\dot u=0\) on \(v=0\).
From $\dot u=v+F(u,v)$ we obtain $F(u,0)\equiv0.$ Comparison with~\eqref{eq:Taylor-Kuznetsov} therefore gives
\(a_{k0}=0\) for all \(k\geq2\).

The formulas for \(b_{20}\) and \(b_{k0}\), \(k\geq3\), follow from
Corollary~\ref{cor:taylor-coefficients}, taking into account the
factorial normalization used in~\eqref{eq:Taylor-Kuznetsov}.
\end{proof}

\subsection{The first normal form coefficients}
For a planar system with canonical nilpotent linear part,
the formulas in \cite[Appendix B]{Kuznetsov2005} give
\begin{equation}\label{eq:Kuznetsov-first}
\begin{aligned}
A_2&=\frac12 b_{20},\qquad
B_2=a_{20}+b_{11},\\
A_3&=
\frac16b_{30}
+\frac12b_{20}a_{11}
-\frac12b_{11}a_{20}.
\end{aligned}
\end{equation}
For the present model,
\[
a_{20}=0,
\qquad
a_{11}
=
-\frac{\eta}{1+\eta(c+x_0)},
\]
and
\[
b_{11}
=
2+4r
-\frac{r(1+c\eta)}{1+\eta(c+x_0)}
-\frac{2(1+r)(cr-1)}{\Delta}.
\]
Consequently,
\begin{equation}\label{eq:Kuznetsov-explicit}
\begin{aligned}
A_2
&=
\frac{r(1+r)^2x_0^2}{\Delta},\\
B_2
&=
2+4r
-\frac{r(1+c\eta)}{1+\eta(c+x_0)}
-\frac{2(1+r)(cr-1)}{\Delta},\\
A_3
&=
\frac{r(1+r)^2x_0}{\Delta^2}
\left(
\Theta-
\frac{\eta x_0\Delta}
     {1+\eta(c+x_0)}
\right).
\end{aligned}
\end{equation}

\begin{corollary}\label{cor:Kuznetsov-beta-zero}
If \(\Theta=0\), then
\[
A_3=\frac12b_{20}a_{11}.
\]
Thus, although the pure cubic Taylor coefficient \(b_{30}\) vanishes,
the cubic normal form coefficient is generally nonzero and is
generated by a mixed quadratic interaction.
\end{corollary}

\begin{proof}
When \(\Theta=0\), Proposition~\ref{prop:Kuznetsov-pure} gives
\(b_{30}=0\). The result follows immediately from
\eqref{eq:Kuznetsov-first}, since \(a_{20}=0\).
\end{proof}

\subsection{The case \(\beta=0\)}
Since \(\Theta=-(1+r)\beta\), the case \(\beta=0\) is equivalent to
\(\Theta=0\). The Bogdanov--Takens parametrization then gives
\[
x_0=\frac{1-cr}{2(1+r)},
\qquad
\Delta=-(1+r)x_0.
\]
Substitution into~\eqref{eq:Kuznetsov-explicit} yields
\begin{equation}\label{eq:Kuznetsov-beta-zero}
\begin{aligned}
A_2
&=
-r(1+r)x_0
=
-\frac{r(1-cr)}{2},\\
B_2
&=
-2-
\frac{r(1+c\eta)}
     {1+\eta(c+x_0)},\\
A_3
&=
\frac{r(1+r)x_0\eta}
     {1+\eta(c+x_0)}.
\end{aligned}
\end{equation}

\begin{corollary}\label{cor:beta-zero}
Assume $r>0$, $c>0$, $\eta>0$, and $cr<1$. Then, for $\beta=0$,
\[
A_2<0,\qquad B_2<0,\qquad A_3>0.
\]
In particular, $A_2B_2\neq0$, and hence the corresponding
double-zero equilibrium is a nondegenerate Bogdanov--Takens
singularity of codimension two.
\end{corollary}

\begin{proof}
Since $cr<1$, the relation
\[
x_0=\frac{1-cr}{2(1+r)}
\]
implies $x_0>0$. The inequalities then follow directly from \eqref{eq:Kuznetsov-beta-zero} and the positivity of
$r$, $c$, and $\eta$. Hence $A_2B_2\neq0$, which is the standard
nondegeneracy condition for a codimension-two Bogdanov--Takens
singularity.
\end{proof}
Thus, within the smooth normal form framework considered here, the
boundary condition $\beta=0$ does not belong to the degenerate
Bogdanov--Takens locus $B_2=0$ in the biologically relevant
parameter regime.


%

\subsection{The degenerate Bogdanov--Takens locus and the fourth-order coefficient}
\label{ssec:cusp}
The codimension-three degenerate Bogdanov--Takens case has already
been analyzed within the recursive normal form framework of Zhao and Zhao
\cite{ZhaoZhao2026}. We now reconsider this degeneracy from the
viewpoint of the smooth normal form of Kuznetsov \cite{Kuznetsov2005}.
Under the nondegeneracy assumption $A_2\neq0$, the additional
condition $B_2=0$ defines the degenerate Bogdanov--Takens locus.
For the present model, after clearing the denominator
$1+\eta(c+x_0)$, the equation $B_2=0$ is linear in $\eta$.
Consequently, once the Bogdanov--Takens parametrization has been
imposed, this equation determines $\eta$ explicitly as a rational
function of $r$, $c$, and $x_0$. This explicit parametrization
allows us to continue the smooth normal form calculation to fourth
order and to identify the coefficient $B_4$ that remains after the
admissible hypernormalizations.

\begin{lemma}\label{lem:eta-cusp}
On the Bogdanov--Takens manifold, assume that
$$-cr+c^2r^2+2cx_0-2rx_0
+5crx_0+5cr^2x_0
+2x_0^2+6rx_0^2+4r^2x_0^2
\neq 0.$$
Then the condition $B_2=0$ is equivalent to
\begin{equation}\label{eq:eta-cusp-explicit}
\eta=\eta_{\mathrm c}(r,c,x_0) = 
-\frac{-r+cr^2+2x_0+5rx_0+3r^2x_0}{-cr+c^2r^2+2cx_0-2rx_0+5crx_0+5cr^2x_0+2x_0^2+6rx_0^2+4r^2x_0^2}.
\end{equation}
\end{lemma}

\begin{proof}
From Proposition~\ref{prop:Kuznetsov-pure}, $B_2=b_{11}$.
Using the explicit expression for $b_{11}$ and clearing the
denominator $1+\eta(c+x_0)$, the equation $B_2=0$ becomes linear
in $\eta$. Under the nonvanishing assumption of the lemma, solving
for $\eta$ yields \eqref{eq:eta-cusp-explicit}. Since the biologically
relevant parameter regime requires $\eta>0$, only values satisfying
$\eta_{\mathrm c}(r,c,x_0)>0$ are admissible.
\end{proof}

Once the condition $B_2=0$ has been imposed, the remaining freedom
in the smooth near-identity transformation may be used successively
to normalize $B_3=0$ and $A_4=0$. After these normalizations, the
fourth-order mixed coefficient $B_4$ provides the next coefficient
relevant to the degenerate Bogdanov--Takens normal form.

\begin{proposition}\label{prop:B4}
After imposing $B_2=0$, $B_3=0$, and $A_4=0$, the fourth-order
mixed coefficient of the Kuznetsov normal form is
\begin{equation}\label{eq:B4-cusp}
B_4=
\frac{P_4(r,c,x_0)}
{r^2x_0^2(c+x_0)\Delta^3},
\end{equation}
where
\[P_4(r,c,x_0)=P^{(0)}(r,c) +P^{(1)}(r,c)x_0 +P^{(2)}(r,c)x_0^2 +P^{(3)}(r,c)x_0^3 +P^{(4)}(r)x_0^4,\]
and the coefficient polynomials $P^{(0)},\ldots,P^{(4)}$
are listed explicitly in Appendix~\ref{app:P4}.
\end{proposition}

\begin{proof}
The remaining freedom in the smooth near-identity transformation is
used to impose the normalization conditions $B_3=0$ and $A_4=0$.
Substituting the resulting expressions, together with
$\eta=\eta_{\mathrm c}(r,c,x_0)$ from
Lemma~\ref{lem:eta-cusp}, into the fourth-order normal form formulas
of Kuznetsov \cite{Kuznetsov2005} yields \eqref{eq:B4-cusp}.
The numerator can be written as a polynomial of degree four in
$x_0$, whose coefficient polynomials are collected in
Appendix~\ref{app:P4}.
\end{proof}

We remark that all identities involved in the derivation of
Proposition~\ref{prop:B4} were independently verified by exact
symbolic computation in \textit{Mathematica}. Consequently, on the
degenerate Bogdanov--Takens locus $B_2=0$, the next higher-order
degeneracy is determined by the single algebraic condition
$P_4(r,c,x_0)=0,$
provided that the denominator in \eqref{eq:B4-cusp} does not vanish.

\appendix

\section{Explicit expression of the polynomial \(P_4(r,c,x_0)\)}
\label{app:P4}
For completeness, we record the coefficient polynomials appearing in
Proposition~\ref{prop:B4}. They determine the numerator of the
fourth-order Kuznetsov coefficient \(B_4\) through the decomposition
\[
P_4(r,c,x_0)
=
P^{(0)}(r,c)
+
P^{(1)}(r,c)x_0
+
P^{(2)}(r,c)x_0^2
+
P^{(3)}(r,c)x_0^3
+
P^{(4)}(r)x_0^4.
\]

The coefficient polynomials are given by
\begin{align*}
P^{(0)}(r,c)
={}&
cr^3
-3c^2r^4
+3c^3r^5
-c^4r^6,
\\[2mm]
P^{(1)}(r,c)
={}&
-7cr^2
+2r^3
-16cr^3
+14c^2r^3
-15cr^4
+32c^2r^4
\\
&
-7c^3r^4
+24c^2r^5
-16c^3r^5
-11c^3r^6,
\\[2mm]
P^{(2)}(r,c)
={}&
16cr
-10r^2
+73cr^2
-16c^2r^2
-25r^3
+144cr^3
\\
&
-73c^2r^3
-15r^4
+143cr^4
-134c^2r^4
+56cr^5
\\
&
-118c^2r^5
-41c^2r^6,
\\[2mm]
P^{(3)}(r,c)
={}&
-12c
+22r
-84cr
+100r^2
-263cr^2
+172r^3
\\
&
-464cr^3+132r^4
-478cr^4
+38r^5
-268cr^5
-63cr^6,
\\[2mm]
P^{(4)}(r)
={}&
-\left(12+90r+271r^2
+425r^3
+369r^4
+169r^5 +32r^6
\right).
\end{align*}

\section*{Acknowledgements}
Martha Alvarez-Ram\'irez was supported by  CBI UAMI research project 2026.

\bibliographystyle{plain}
\bibliography{TBref}
\end{document}